\documentclass[11pt,reqno]{amsart}
\usepackage[T1]{fontenc}
\usepackage{lmodern}
\usepackage{amsmath,amssymb,mathtools}
\usepackage{mathrsfs}
\usepackage{microtype}
\usepackage{graphicx}
\usepackage{enumitem}
\usepackage[a4paper,margin=28mm]{geometry}
\usepackage{xcolor}
\definecolor{linkblue}{RGB}{0,0,255}
\usepackage[colorlinks=true,linkcolor=linkblue,citecolor=linkblue,urlcolor=linkblue]{hyperref}
\numberwithin{equation}{section}
\newtheorem{theorem}{Theorem}[section]
\newtheorem{lemma}[theorem]{Lemma}
\newtheorem{assumption}[theorem]{Assumption}
\newtheorem{proposition}[theorem]{Proposition}
\newtheorem{corollary}[theorem]{Corollary}
\theoremstyle{definition}
\newtheorem{definition}[theorem]{Definition}
\theoremstyle{remark}
\newtheorem{remark}[theorem]{Remark}

\newcommand{\Tmax}{T_{\max}}

\setlist[enumerate]{leftmargin=*,label=\textup{(\roman*)}}
\allowdisplaybreaks[1]

\newcommand{\dd}{\,\mathrm{d}}

\title[Blow-up for wave equations with nonlinear damping]{Finite-Time Blow-Up for Wave Equations with Nonlinear \(p\)-Laplacian Damping}
\author[M.  Liao]{Menglan Liao}
\address{(M. Liao)School of Mathematics, Hohai University, Nanjing, 210098, China \&
Laboratory of Mathematical Modeling and Intelligent Computing for Water Systems, Hohai University, Nanjing, 211100, China.}
 \email{liaoml@hhu.edu.cn}
 
\author[M. Zhang]{Mingxue Zhang}
\address{(M. Zhang)Institute for Math and AI, Wuhan University, Wuhan, 430072, China. }
 \email{zhangmingxuex@163.com}
  
\thanks{This work was supported by the National Natural Science Foundation of China (12401290) and the Natural Science Foundation of Jiangsu Province (BK20230946). }

\subjclass[2020]{Primary 35B44; Secondary 35L72, 35D30.}
\keywords{nonlinear strong damping; \(p\)-Laplacian; energy weak solution;
finite-time blow-up; potential well}
\date{}

\begin{document}
\begin{abstract}
We consider a wave equation with linear elasticity, nonlinear \(p\)-Laplacian damping acting on the velocity, and a focusing power source in a bounded three-dimensional domain. For \(2<p<r+1\leq6\), we prove finite-time blow-up in the natural energy norm for initial data in \(H_0^1(\Omega)\times L^2(\Omega)\) lying in the unstable potential well below its depth. This includes every nonzero initial state with nonpositive total energy and requires no sign condition on the initial displacement--velocity pairing. For positive initial energy, we establish a sufficient condition for finite-time blow-up involving the \(L^2\) pairing of the initial velocity with a fixed approximation of the initial displacement. 
The proof combines a shifted displacement multiplier, an auxiliary quantity built from the time integral of the spatial \(L^p\)-norm of the velocity gradient, and a weighted auxiliary functional.
Linking the weight parameter to the initial value of the cumulative quantity allows us to absorb the mixed terms in the auxiliary functional estimate. Together with the shifted displacement multiplier, this yields upper bounds for the maximal existence time in both ranges \(2<p\leq3\) and \(p>3\), without assuming \(W_0^{1,p}(\Omega)\) regularity of the initial displacement.
\end{abstract}
\maketitle

\section{Introduction}
We study finite-time blow-up for a wave equation with nonlinear
strong damping and a focusing source. Let
$\Omega\subset\mathbb R^3$ be a bounded domain with $C^2$ boundary.
Consider
\begin{equation}\label{eq1}
 \begin{cases}
 u_{tt}-\Delta u-\Delta_pu_t=f(u),
       &x\in\Omega,\quad t>0,\\
 u=0,&x\in\partial\Omega,\quad t>0,\\
 u(0)=u_0,\quad u_t(0)=u_1,&x\in\Omega,
 \end{cases}
\end{equation}
where
\[
 \Delta_pw=\operatorname{div}(|\nabla w|^{p-2}\nabla w),
 \quad p'=\frac{p}{p-1}.
\]
After recalling the local theory, we consider the power source
\[
 f(u)=|u|^{r-1}u,\quad
 2<p<r+1\leq6,\quad
 (u_0,u_1)\in H_0^1(\Omega)\times L^2(\Omega).
\]
The condition $r+1\leq6$ ensures that the source potential
is finite on $H_0^1(\Omega)$.
For $3<r\leq5$, the source is supercritical with respect to
local Lipschitz continuity from $H_0^1(\Omega)$ into $L^2(\Omega)$.
The endpoint $r+1=6$ is critical for the Sobolev embedding
of the displacement energy space.

Strongly damped wave equations arise naturally in the mathematical
description of viscoelastic materials, where the stress depends on
both the strain and its rate of change. 
A basic example is
\[
 u_{tt}-\Delta u-\alpha\Delta u_t=f(u),
 \quad \alpha>0,
\]
with homogeneous Dirichlet boundary conditions.
The Kelvin--Voigt term $-\alpha\Delta u_t$ dissipates energy
through the velocity gradient.
Webb \cite{Webb1980} used analytic semigroup theory to prove
global existence and uniqueness of strong solutions under
suitable dissipativity assumptions.
For defocusing nonlinearities, Ang and Dinh
\cite{AngDinh1988} used successive linearizations to prove
local existence and uniqueness, and Galerkin and energy
estimates to obtain global existence and exponential decay
under stronger assumptions. 
In contrast, for focusing sources the solution may blow up in finite time, and two standard approaches to studying global existence versus blow-up are the potential well method and the concavity method. 
Sattinger \cite{Sattinger1968} and Payne and Sattinger
\cite{PayneSattinger1975} used variational thresholds and
invariant sets to study global existence and instability.
Levine \cite{Levine1974a,Levine1974b} used concavity
inequalities to prove finite-time blow-up.
Levine and Serrin \cite{LevineSerrin1997} developed
an abstract nonexistence theory for quasilinear evolution
equations with dissipation.

For semilinear wave equations with linear strong and weak
damping, Gazzola and Squassina \cite{GS2006} combined
potential well arguments, energy estimates, and the
concavity method to study boundedness of global solutions,
global existence for stable data, and blow-up for suitable
unstable data, including data with arbitrarily large energy.
Zu, Guo, and Gao \cite{ZGG2022} developed energy estimates
and obtained decay, high-energy blow-up, and lifespan bounds.
Linear strong damping has also been studied with nonlinear
elasticity. Pei, Rammaha, and Toundykov \cite{PRT2015}
considered
\[
 u_{tt}-\Delta_pu-\Delta u_t=f(u),
 \quad 2<p<3,
\]
in a bounded three-dimensional domain with homogeneous
Dirichlet boundary conditions.
They used Galerkin approximations and compactness to
construct local weak solutions for supercritical sources.
They also proved global existence under suitable source
restrictions and blow-up for negative initial energy.
Kass and Rammaha \cite{KR2019} used Galerkin approximations
and monotone operator theory to construct weak solutions
with interior and boundary sources, and obtained
global-existence and negative-energy blow-up results.
For logarithmic sources, Yang and Han \cite{YH2022}
combined Nehari-manifold and concavity arguments to prove
blow-up below the potential-well depth. They also obtained
blow-up criteria above that threshold and upper and lower
bounds for the blow-up time.
For power sources, Zu, Sun, and Wu \cite{ZSW2023}
derived energy estimates and established global-existence
and blow-up results, including blow-up for suitable data
with arbitrarily large positive energy and lifespan estimates.

Linear strong damping has two useful properties.
Testing against $u_t$ gives the quadratic dissipation
$\|\nabla u_t\|_2^2$, while testing against $u$ gives
\[
 \langle-\Delta u_t(t),u(t)\rangle
 =\frac12\frac{\mathrm d}{\mathrm dt}
       \|\nabla u(t)\|_2^2
 \quad\text{for a.e. }t\in(0,T),
\]
under the corresponding energy regularity.
Thus the damping contribution becomes an exact time
derivative, which supports energy and concavity estimates.
Its linearity also avoids an additional nonlinear weak
limit in approximation schemes for nonlinear elasticity.

For problem \eqref{eq1}, Rammaha and Wilstein
\cite{RW2012} used monotone operator theory and approximation
arguments to establish local existence, uniqueness,
continuous dependence, and the energy identity.
They also proved global existence when the damping
exponent dominates the source growth.
Rammaha, Toundykov, and Wilstein \cite{RTW2012}
used Nehari-manifold analysis and energy estimates to
prove global existence in the stable potential well,
beyond that growth restriction, and derived energy decay
under additional assumptions.
These works provide the local and variational frameworks
used here for the same Dirichlet problem. 
The results in \cite{RTW2012,RW2012} establish
well-posedness and stabilization, but do not provide
blow-up criteria for this Dirichlet problem with a
focusing power source. 
Two difficulties arise in proving blow-up for problem 
\eqref{eq1}: 
\begin{enumerate}
  \item First, the elastic energy controls $\|\nabla u\|_2$,
whereas the damping requires gradients in $L^p(\Omega)$.
Since $H_0^1(\Omega)$ is not contained in
$W_0^{1,p}(\Omega)$ for $p>2$, the displacement $u$
is not automatically an admissible test function.
  \item Second, even when testing against $u$ is justified,
the damping contribution is
\[
 \int_\Omega |\nabla u_t|^{p-2}
       \nabla u_t\cdot\nabla u\,\mathrm dx.
\]
This term is not a time derivative of the quadratic
elastic energy. Its direct estimate involves
$\|\nabla u_t\|_p^{p-1}\|\nabla u\|_p$,
but the natural energy does not control
$\|\nabla u\|_p$.
The dissipation controls
$\int_0^t\|\nabla u_s\|_p^p\,\mathrm ds$,
which alone does not provide the missing spatial
regularity for rough initial displacements.
\end{enumerate}

Additional $W_0^{1,p}(\Omega)$ regularity of $u_0$
resolves the testing issue but still leaves the nonlinear
damping term to be estimated.
Our aim is to address both difficulties and establish
blow-up criteria throughout $2<p<r+1\leq6$ while retaining
the initial-data space
$H_0^1(\Omega)\times L^2(\Omega)$. 
The proof uses three ingredients:
\begin{enumerate}
  \item First, we introduce the shifted displacement
$v(t)=u(t)-u_0+v_0$, where $v_0$ is a fixed smooth
approximation of $u_0$.
The regularity of $u(t)-u_0$ makes $v$ an admissible
test function, and the approximation error is controlled
in $H_0^1(\Omega)$ and $L^{r+1}(\Omega)$.
  \item Second, we introduce
$\mathcal R(t)=\mathcal R_0+
\int_0^t\|\nabla u_s(s)\|_p\,\mathrm ds$,
where $\mathcal R_0\geq\max\{1,\|\nabla v_0\|_p\}$.
This quantity controls $\|\nabla v(t)\|_p$ and
expresses the damping estimate in terms of
$\mathcal R(t)$ and $\mathcal R'(t)$,
with $(\mathcal R'(t))^p$ equal to the energy dissipation rate.
  \item Third, we construct a weighted auxiliary functional
from a shifted negative energy, the pairing $(v,u_t)$,
and a correction depending on $\mathcal R(t)$.
Choosing the weight in relation to $\mathcal R_0$
allows the mixed terms to be absorbed for both
$2<p\leq3$ and $p>3$.
Once its initial positivity is established,
a superlinear differential inequality yields
finite-time blow-up and explicit lifespan bounds.

\end{enumerate}

The argument gives three blow-up results.
First, negative total energy implies finite-time
blow-up without a sign condition on $(u_0,u_1)$.
Second, the same conclusion holds for initial data in
the unstable potential-well set $\mathscr W_2$
with $E(0)<d$, including $0\leq E(0)<d$
without an additional condition on the initial pairing.
Third, for arbitrary positive initial energy,
we obtain a sufficient blow-up criterion involving
$(v_0,u_1)$ and the approximation error $u_0-v_0$.
All three results apply to data in
$H_0^1(\Omega)\times L^2(\Omega)$ satisfying the
respective hypotheses and give explicit upper bounds
for the maximal existence time.

The paper is organized as follows.
Section~\ref{sec2} recalls the weak-solution framework, energy
identity, and continuation alternative, introduces the
potential well, and states the main results.
Section~\ref{sec3} develops the shifted multiplier and auxiliary
functional estimates and proves the three blow-up theorems.
The appendix provides the product rule for weak solutions
and derives the continuation alternative from the local
construction and uniqueness.

\section{Preliminaries and main results}\label{sec2}

\subsection{Weak solutions, energy identity, and continuation}

We write \(\|w\|_\ell=\|w\|_{L^\ell(\Omega)}\) and
\((w,z)=(w,z)_{L^2(\Omega)}\). The brackets
\(\langle\cdot,\cdot\rangle\) denote the duality pairing
between a Sobolev space and its dual, with the dual element
written first. By Poincar\'e's inequality, we use the equivalent
norms \(\|\nabla w\|_2\) on \(H_0^1(\Omega)\) and
\(\|\nabla w\|_p\) on \(W_0^{1,p}(\Omega)\).

\begin{assumption}[Finite-energy local framework]\label{ass:local}
Assume \(2<p<\infty\), \(1\leq r\leq5\), and \(f\in C^1(\mathbb R)\).
For \(|s|\geq1\), suppose
\[
 |f(s)|\leq b_0|s|^r,\quad
 |f'(s)|\leq b_1|s|^{r-1},
\]
where \(b_0,b_1>0\). If \(3<r\leq5\), assume also
\(f\in C^2(\mathbb R)\) and
\(|f''(s)|\leq b_2|s|^{r-2}\) for \(|s|\geq1\).
The initial data belong to \(H_0^1(\Omega)\times L^2(\Omega)\).
\end{assumption}
This is the part of \cite[Assumption~1.1]{RW2012} needed here.
For \(2<p<3\), its source restriction is automatic because
\(r\leq5<8-6/p\); for \(p\geq3\), one has \(r\leq5<6\).
The additional regularity of \(f\) is used in the cited uniqueness theory.
In this range, \(|F(s)|\leq C(1+|s|^{r+1})\), so the total energy is
finite for every initial state under consideration.

\begin{definition}[Weak solution]\label{def:weak}
A function $u$ is said to be a weak solution of problem \eqref{eq1} on $[0,T]$ if $u\in C([0,T],\mathrm{H}_0^1(\Omega))$, $u_t\in C([0,T],L^2(\Omega))\cap L^p(0,T,\mathrm{W}_0^{1,p}(\Omega))$, $(u(0),u_t(0))=(u_0,u_1)$ and, for all $0\le t\le T$, $u$ satisfies:
\begin{equation}
\label{eq2}
\begin{split}
&\int_0^t\int_\Omega \big(-u_t\varphi_t+\nabla u\cdot\nabla\varphi\big)\,dxds
+\left.\int_\Omega u_t\varphi\right|_{s=0}^{s=t}dx\\
&\quad+\int_0^t\int_\Omega |\nabla u_t|^{p-2}\nabla u_t\cdot\nabla\varphi\,dxds
=\int_0^t\int_\Omega f(u)\varphi\,dxds,
\end{split}
\end{equation}
for all test functions $\varphi\in H^1(0,T;L^2(\Omega))\cap L^p(0,T;\mathrm{W}_0^{1,p}(\Omega))$.
\end{definition}

We use the following consequences of
\cite[Theorems~1.2--1.3]{RW2012}.

\begin{theorem}[Local theory in the finite-energy range]\label{thm:local}
Under Assumption~\ref{ass:local}, problem \eqref{eq1} has a unique
local weak solution on \([0,T]\) for some \(T>0\).
It satisfies
\begin{equation}\label{eq3}
 E(t)+\int_{t_0}^t\|\nabla u_t(s)\|_p^p\,\dd s=E(t_0),
 \quad 0\leq t_0\leq t\leq T,
\end{equation}
where
\begin{equation}\label{eq4}
 E(t)=\mathcal E(t)-\int_\Omega F(u(t))\,\dd x,\quad
 F(s)=\int_0^s f(\sigma)\,\dd\sigma,
\end{equation}
and
\[
 \mathcal E(t)=\frac12\|u_t(t)\|_2^2+
                     \frac12\|\nabla u(t)\|_2^2.
\]
If \(r\leq p-1\), the solution is global.
\end{theorem}

Whenever the initial energy is finite,
\eqref{eq3} and
\(\|\nabla u_t\|_p^p\in L^1(0,T)\) imply that
\(E\) is absolutely continuous on \([0,T]\), with
\[
 E'(t)=-\|\nabla u_t(t)\|_p^p, \quad\text{a.e. }t\in(0,T).
\]

The continuation result needed below concerns \(r\leq5\),
where both finite energy and uniqueness hold for all data in
\(H_0^1(\Omega)\times L^2(\Omega)\).

\begin{corollary}[Continuation alternative]\label{cor:continuation}
Under Assumption~\ref{ass:local}.
The local solution in Theorem~\ref{thm:local} admits a unique
maximal continuation \(u\) on \([0,\Tmax)\), where
\[
 \Tmax=\sup\{T>0:\text{a weak solution with data }(u_0,u_1)
                         \text{ exists on }[0,T]\}
 \in(0,\infty].
\]
It satisfies the energy identity on every compact subinterval
of \([0,\Tmax)\). Moreover,
\begin{equation}\label{eq5}
 \Tmax<\infty\quad\Longrightarrow\quad
 \limsup_{t\uparrow\Tmax}\mathcal E(t)
 =\limsup_{t\uparrow\Tmax}
 \frac12\left(\|u_t(t)\|_2^2+\|\nabla u(t)\|_2^2\right)
 =+\infty.
\end{equation}
\end{corollary}

The proof, based on a uniform local existence time on bounded
sets of initial data and uniqueness, is given in the appendix.

From this point onward, we take \(f(s)=|s|^{r-1}s\) and assume
\begin{equation}\label{eq6}
 2<p<r+1\leq6,\quad
 (u_0,u_1)\in H_0^1(\Omega)\times L^2(\Omega).
\end{equation}
The power source satisfies Assumption~\ref{ass:local}:
it is \(C^1\) because \(r>1\), and it has the required \(C^2\)
regularity when \(r>3\). Thus Theorem~\ref{thm:local} and
Corollary~\ref{cor:continuation} apply to every state in
\eqref{eq6}.

The total energy in \eqref{eq4} becomes
\[
E(t)=\frac12\|u_t(t)\|_2^2+\frac12\|\nabla u(t)\|_2^2
             -\frac1{r+1}\|u(t)\|_{r+1}^{r+1}, \quad \forall t\in[0,\Tmax).
\]
The inclusion \(W_0^{1,p}(\Omega)\hookrightarrow H_0^1(\Omega)\)
follows from
\[
 \|\nabla w\|_2
 \leq|\Omega|^{1/2-1/p}\|\nabla w\|_p.
\]
Together with \(H_0^1(\Omega)\hookrightarrow L^{r+1}(\Omega)\),
this gives constants \(C_P,C_E>0\), depending only on
\(p,r,\Omega\), such that
\begin{equation}\label{eq7}
 \|w\|_2\leq C_P\|\nabla w\|_p,\quad
 \|w\|_{r+1}\leq C_E\|\nabla w\|_p,
 \quad w\in W_0^{1,p}(\Omega).
\end{equation}
These constants remain fixed throughout the proof.

\subsection{Potential well}

We use the potential well notation of \cite[Section~3]{RTW2012}.
For \(w\in H_0^1(\Omega)\), define
\begin{equation*}
J(w)=\frac12\|\nabla w\|_2^2
       -\frac1{r+1}\|w\|_{r+1}^{r+1}.
\end{equation*}
Then \(E(t)=\frac12\|u_t(t)\|_2^2+J(u(t))\) for all $t\in[0,\Tmax)$, and
\[
 \langle J'(w),w\rangle
 =\|\nabla w\|_2^2-\|w\|_{r+1}^{r+1}.
\]
The Nehari manifold is
\begin{equation*}
\begin{split}
 \mathcal N
 &=\{w\in H_0^1(\Omega)\setminus\{0\}:
                    \langle J'(w),w\rangle=0\}\\
 &=\{w\in H_0^1(\Omega)\setminus\{0\}:
                    \|\nabla w\|_2^2=\|w\|_{r+1}^{r+1}\}.
 \end{split}
\end{equation*}
The depth is
\begin{equation}\label{eq8}
 d=\inf_{w\in H_0^1(\Omega)\setminus\{0\}}
           \sup_{\tau\geq0}J(\tau w)
   =\inf_{w\in\mathcal N}J(w)>0.
\end{equation}
Following \cite{RTW2012}, set
\begin{equation*}
\begin{split}
 \mathscr W
 &=\{w\in H_0^1(\Omega):J(w)<d\},\\
 \mathscr W_1
 &=\{w\in\mathscr W:
       \|\nabla w\|_2^2>\|w\|_{r+1}^{r+1}\}\cup\{0\},\\
 \mathscr W_2
 &=\{w\in\mathscr W:
       \|\nabla w\|_2^2<\|w\|_{r+1}^{r+1}\}.
 \end{split}
\end{equation*}
Thus \(\mathscr W\) is the potential well, \(\mathscr W_1\)
is its stable part, and \(\mathscr W_2\) is its unstable part. 

Also following \cite[Section~3.1]{RTW2012}, define
\begin{equation}\label{eq9}
 R=\sup_{w\in H_0^1(\Omega)\setminus\{0\}}
       \frac{\|w\|_{r+1}^{r+1}}{\|\nabla w\|_2^{r+1}}
 \in(0,\infty).
\end{equation}
In particular,
\begin{equation}\label{eq10}
 \|w\|_{r+1}\leq R^{1/(r+1)}\|\nabla w\|_2.
\end{equation}
The associated scalar function and its maximizer are
\begin{equation}\label{eq11}
 \mathcal G(s)=\frac12s^2-\frac{R}{r+1}s^{r+1},
 \quad s\geq0,\quad s_0=R^{-1/(r-1)}.
\end{equation}
Then \(J(w)\geq\mathcal G(\|\nabla w\|_2)\).
For the pure-power potential considered here, the optimal constant
\(R\) defined in \eqref{eq9} determines the potential-well
depth exactly:
\begin{equation}\label{eq12}
 d=\widetilde d=\mathcal G(s_0)
   =\frac{r-1}{2(r+1)}R^{-2/(r-1)}
   =\frac{r-1}{2(r+1)}s_0^2>0.
\end{equation}
Here we retain the notation \(\widetilde d=\mathcal G(s_0)\)
for consistency with \cite[Section~3.1]{RTW2012}.
The identity is proved in Section~\ref{sec3.2} directly from
the definition of \(R\), without assuming that the supremum
in \eqref{eq9} is attained.

\subsection{Main results}

We specify the parameters entering the lifespan bounds for the
first two theorems. Given \(\nu>p\), \(h_0>0\), and
\(v_0\in C_c^\infty(\Omega)\), choose
\begin{equation}\label{eq13}
 p<\lambda<\nu,\quad (3-p)\lambda<2,\quad
 \mu=\frac{\lambda+\nu}{2},\quad
 \beta=\frac{\lambda-p}{p-1},\quad c_0=1+\frac\nu2.
\end{equation}
Such a \(\lambda\) exists whenever \(\nu>p>2\). Indeed, if \(p\geq3\), the second
condition is automatic. If \(2<p<3\), then
\[
 \frac2{3-p}-p=\frac{(p-1)(p-2)}{3-p}>0,
\]
so one can choose
\[
 p<\lambda<\min\left\{\nu,\frac2{3-p}\right\}.
\]

Set
\begin{equation}\label{eq14}
 m=\inf_{z\geq0}
 \frac{(z-1)(z^{p-1}-1)+\mu/\lambda-1}{1+z^2}>0,\quad
 \widehat{C}=\frac{C_P^2}{c_0}(\beta^2+\mu^2),\quad
 \kappa=\left(\frac{\lambda h_0}{4}\right)^{1/(p-1)}.
\end{equation}
The positivity of \(m\) is proved in
Lemma~\ref{lem:functional}. The parameters $\mathcal R_0$ and $a$ can be chosen
explicitly as follows:
\begin{equation}\label{eq15}
 \begin{split}
 \mathcal{R}_0=\max\biggl\{1,\|\nabla v_0\|_p,
 \left(\frac{4\kappa|(v_0,u_1)|}{h_0}\right)^{(p-1)/p},\,
 \left(\frac{2\widehat{C}}{m}\right)^{(p-1)/2}
                    \kappa^{(3-p)(p-1)/2}\biggr\},\quad
 a=\kappa \mathcal{R}_0^{-\lambda/(p-1)}.
 \end{split}
\end{equation}
All these parameters are fixed before the solution is estimated.
The quantities \(\nu,h_0,v_0\) and the upper-bound constant
\(M\) are specified separately in each theorem. 
For Theorems~\ref{thm:negative} and \ref{thm:well} below,  first choose
\(\varepsilon,\delta\), hence \(\nu\), then fix \(v_0\) and \(h_0\) as
specified in the theorem. Only afterwards choose \(\lambda,\mu,\beta\),
the constants \(m,\widehat C,\kappa\), and finally \(\mathcal R_0,a\).
Increasing \(\mathcal R_0\) does not alter the fixed approximation error.
All prescriptions remain valid if \(m\) is replaced consistently by the
explicit smaller constant
\[
 m_*=\frac15\min\left\{\frac\mu\lambda-1,\ 2^{p-1}-1\right\}>0, 
\]
see the proof of Lemma~\ref{lem:functional}. This avoids evaluating the
one-dimensional infimum in \eqref{eq14}.

\begin{theorem}[Negative initial energy]\label{thm:negative}
Assume \eqref{eq6}. If \(E(0)<0\), then the maximal weak
solution of \eqref{eq1} has finite existence time and
satisfies \eqref{eq5}. 

More precisely, choose \(\varepsilon\in(0,1)\) and
\(\delta>0\) such that
\[
 \nu=(1-\varepsilon)(r+1)>p,\quad
 \frac\nu2-1-\delta>0.
\]
Choose \(v_0\in C_c^\infty(\Omega)\), set
\(\eta=u_0-v_0\), and require
\[
 C_\eta:=\frac{\|\nabla\eta\|_2^2}{4\delta}
 +\frac{r^r\|\eta\|_{r+1}^{r+1}}
 {(r+1)^{r+1}\varepsilon^r}
 <\frac\nu2(-E(0)).
\]
Define
\[
 h_0=-E(0)-\frac{C_\eta}{\nu}>0,\quad
 M=\frac{2}{r+1}(C_E+\|\eta\|_{r+1})^{r+1},
\]
and choose the remaining parameters by
\eqref{eq13}--\eqref{eq15}.
Then the maximal existence time satisfies
\begin{equation}\label{eq16}
 \Tmax\leq\frac{r+1}{a\mu\beta}
 \left(\frac{2(M+aC_P\sqrt M)}{h_0}\right)^{\beta/(r+1)}.
\end{equation}
\end{theorem}

\begin{theorem}[Blow-up in \(\mathscr W_2\) below the well depth]\label{thm:well}
Assume \eqref{eq6}. If
\begin{equation*}
E(0)<d,\quad u_0\in\mathscr W_2,
\end{equation*}
then \(\Tmax<\infty\) and \eqref{eq5} holds.

For a quantitative bound, let \(\tilde{s}>s_0\)
be the unique solution on the decreasing branch of \(\mathcal G\)
of
\[
 \mathcal G(\tilde{s})=\frac12\tilde{s}^2-\frac{R}{r+1}\tilde{s}^{r+1}=E(0).
\]
Choose \(\varepsilon\in(0,1)\) and \(\delta>0\) so that
\begin{equation}
\label{eq17}
 \nu=(1-\varepsilon)(r+1)>p,\quad
 \left(\frac\nu2-1-\delta\right)\tilde{s}^2-\nu d>0.
\end{equation}
Choose \(v_0\in C_c^\infty(\Omega)\) so that, with
\(\eta=u_0-v_0\) and \(C_\eta\) defined as in
Theorem~\ref{thm:negative},
\[
 C_\eta<\left(\frac\nu2-1-\delta\right)\tilde{s}^2-\nu d.
\]
Set
\[
 h_0=d-E(0),\quad
 M=d+2\max\{E(0),0\}
       +\frac{2}{r+1}(C_E+\|\eta\|_{r+1})^{r+1},
\]
and use \eqref{eq13}--\eqref{eq15}.
Then
\begin{equation}\label{eq18}
 \Tmax\leq\frac{r+1}{a\mu\beta}
 \left(\frac{2(M+aC_P\sqrt M)}{d-E(0)}\right)^{\beta/(r+1)}.
\end{equation}
\end{theorem}

\begin{corollary}[Nonpositive initial energy]\label{cor:nonpositive}
Under \eqref{eq6}, every nonzero initial state with
\(E(0)\leq0\) generates a solution satisfying the conclusion
of Theorem~\ref{thm:well}.
\end{corollary}

\begin{theorem}[A sufficient blow-up criterion at positive initial energy]\label{thm:positive}
Assume \eqref{eq6} and \(E(0)>0\).
First choose \(\varepsilon\in(0,1)\), \(\delta>0\) so that
\[
 \nu=(1-\varepsilon)(r+1)>p,\quad \frac\nu2-1-\delta>0.
\]
Choose \(\lambda,\mu,\beta,c_0\) by \eqref{eq13},
and let \(m,\widehat C\) have the values in
\eqref{eq14}. Define
\begin{equation}\label{eq19}
 \widetilde C=
 \left(\frac{2\widehat C}{m}\right)^{(p-1)/2}(2\lambda)^{(3-p)/2}.
\end{equation}
Fix \(v_0\in C_c^\infty(\Omega)\), let \(\eta=u_0-v_0\), and define
\(C_\eta\) as in Theorem~\ref{thm:negative}.
Suppose there is \(\mathcal R_0\) such that
\begin{equation}\label{eq20}
 \mathcal R_0\geq
 \max\left\{1,\|\nabla v_0\|_p,\
       \widetilde C\left(E(0)+\frac{C_\eta}{\nu}\right)^{(3-p)/2}\right\}
\end{equation}
and
\begin{equation}\label{eq21}
 (v_0,u_1)>
 \frac{3}{(2\lambda)^{1/(p-1)}}
 \left(E(0)+\frac{C_\eta}{\nu}\right)^{(p-2)/(p-1)}
                 \mathcal R_0^{p/(p-1)}.
\end{equation}
Then \(\Tmax<\infty\) and \eqref{eq5} holds.

In addition, for these fixed choices, put
\begin{equation}\label{eq22}
 \kappa=\left[2\lambda\left(E(0)+\frac{C_\eta}{\nu}\right)\right]^{1/(p-1)},
 \quad a=\kappa\mathcal R_0^{-\lambda/(p-1)},
\end{equation}
and
\[
 M=2E(0)+\frac{2}{r+1}(C_E+\|\eta\|_{r+1})^{r+1}.
\]
Then
\begin{equation}\label{eq23}
 \Tmax\leq\frac{r+1}{a\mu\beta}
 \left(
 \frac{M+aC_P\sqrt M}
 {\kappa\mathcal R_0^{-p/(p-1)}(v_0,u_1)
           -3(E(0)+C_\eta/\nu)}
 \right)^{\beta/(r+1)}.
\end{equation}
\end{theorem}

Fix the initial data \((u_0,u_1)\), the approximation \(v_0\),
and admissible choices of \(\varepsilon,\delta,\lambda,\mu\).
Set
\[
 \mathcal R_{\min}
 =\max\left\{
 1,\|\nabla v_0\|_p,\,
 \widetilde C
 \left(E(0)+\frac{C_\eta}{\nu}\right)^{(3-p)/2}
 \right\}.
\]
Then \eqref{eq20} requires
\(\mathcal R_0\geq\mathcal R_{\min}\).
For these fixed choices, the right-hand side of
\eqref{eq21} is strictly increasing in
\(\mathcal R_0\).
Thus, among the admissible values of \(\mathcal R_0\),
the choice \(\mathcal R_0=\mathcal R_{\min}\) gives the least
restrictive pairing condition.

With \(a\) chosen according to \eqref{eq22},
the pairing condition \eqref{eq21} is
equivalent to
\[
 G(0)
 =\kappa\mathcal R_0^{-p/(p-1)}(v_0,u_1)
 -3\left(E(0)+\frac{C_\eta}{\nu}\right)>0.
\]
The theorem gives a sufficient condition for finite-time
blow-up; failure of this condition does not imply global
existence.
Remark~\ref{rem:regular-data} gives the corresponding
criterion when \(u_0\in W_0^{1,p}(\Omega)\), for which one
may take \(v_0=u_0\) and \(C_\eta=0\).
Remark~\ref{rem:positive-class} shows that the class of
initial data satisfying the criterion is nonempty.

\begin{remark}[Initial displacement in the damping space]
\label{rem:regular-data}
In each of Theorems~\ref{thm:negative}, \ref{thm:well}, and
\ref{thm:positive}, if additionally
\(u_0\in W_0^{1,p}(\Omega)\), one may take directly
\[
v_0=u_0,\quad \eta=0,\quad C_\eta=0,
\]
and 
\[
v(t)=u(t),\quad Y(t)=(u(t),u_t(t)),\quad \forall t\in[0,\Tmax).\]
Indeed, all estimates for the shifted test function require
only \(v_0\in W_0^{1,p}(\Omega)\). For the direct choice above, 
\[
\begin{aligned}
 u(t)=u_0+\int_0^t u_s(s)\dd s,\quad \forall t\in[0,\Tmax),
\end{aligned}
\]
and 
\[
\begin{aligned}
 \|\nabla u(t)\|_p
\leq\|\nabla u_0\|_p+\int_0^t\|\nabla u_s(s)\|_p\dd s
 \leq\mathcal R(t),\quad \forall t\in[0,\Tmax),
\end{aligned}
\]
provided \(\mathcal R_0\geq\max\{1,\|\nabla u_0\|_p\}\).
Hence \(u\) is an admissible test function and the
approximation error vanishes. 
The energy assumptions and any additional condition on
the initial displacement--velocity pairing required by
the respective theorem still apply. In particular, in
Theorem~\ref{thm:positive}, \(\big(E(0)+\frac{C_\eta}{\nu}\big)=E(0)\) and
\eqref{eq21} reduces to
\[
 (u_0,u_1)>
 \frac{3}{(2\lambda)^{1/(p-1)}}
 (E(0))^{(p-2)/(p-1)}\mathcal R_0^{p/(p-1)}.
\]
\end{remark}

\section{Proofs of the main results}\label{sec3}
\begin{lemma}[Shifted displacement and multiplier identity]
\label{lem:shifted-multiplier}
Assume \eqref{eq6}, and let \(u\) be the maximal weak
solution of problem \eqref{eq1}.
Fix \(T\in(0,\Tmax)\) and \(v_0\in W_0^{1,p}(\Omega)\).
Define
\begin{equation}\label{eq24}
 \eta=u_0-v_0,\quad
 v(t)=u(t)-\eta
     =v_0+\int_0^t u_s(s)\dd s,
 \quad \forall t\in[0,T].
\end{equation}
Then \(v\) is an admissible test function
in \eqref{eq2}. Moreover, the function
\begin{equation}\label{eq25}
 Y(t):=(v(t),u_t(t)),\quad \forall t\in[0,T],
\end{equation}
is absolutely continuous on \([0,T]\) and satisfies
\begin{equation}\label{eq26}
\begin{aligned}
 Y'(t)
 ={}&\|u_t(t)\|_2^2-\|\nabla u(t)\|_2^2
       +\|u(t)\|_{r+1}^{r+1}
       +(\nabla u(t),\nabla\eta)\\
 &-\int_\Omega |u(t)|^{r-1}u(t)\eta\dd x
 -\int_\Omega |\nabla u_t(t)|^{p-2}
             \nabla u_t(t)\cdot\nabla v(t)\dd x
\end{aligned}
\end{equation}
for almost every \(t\in(0,T)\).
\end{lemma}

\begin{proof}
By \cite[Section~3.1]{RW2012}, we have
\begin{equation}\label{eq27}
 u_{tt}\in L^{p'}(0,T;W^{-1,p'}(\Omega)).
\end{equation}

Set
\[
 X=W_0^{1,p}(\Omega),\quad
 H=L^2(\Omega),\quad
 X^*=W^{-1,p'}(\Omega).
\]
Since \(u_t\in L^p(0,T;X)\), the Bochner integral
\[
 \widetilde v(t)=v_0+\int_0^t u_s(s)\dd s
\]
defines a function in \(W^{1,p}(0,T;X)\), with
\(\widetilde v_t=u_t\).
Its continuous \(X\)-valued representative satisfies
\(\widetilde v(0)=v_0\).
On the other hand, the weak-solution regularity gives
\[
 u(t)-u_0=\int_0^t u_s(s)\dd s
 \quad\text{in }H,\quad \forall t\in[0,T].
\]
Consequently, \(\widetilde v(t)=u(t)-\eta=v(t)\) for
every \(t\in[0,T]\). Thus
\[
 v\in W^{1,p}(0,T;X)\cap C([0,T];X),\quad
 v(0)=v_0,\quad
 v_t=u_t\quad\text{a.e. in }(0,T).
\]
Since \(p>2\), \(T<\infty\), and
\(X\hookrightarrow H\) continuously, we also have
\(v\in H^1(0,T;H)\).
Therefore \(v\) belongs to the test-function space
\(H^1(0,T;L^2(\Omega))
 \cap L^p(0,T;W_0^{1,p}(\Omega))\)
in \eqref{eq2}.

The preceding regularity and that of the weak solution give
\[
 u_t,v\in L^p(0,T;X)\cap C([0,T];H)
       \subset L^p(0,T;X)\cap L^2(0,T;H).
\]
Since \(v_t=u_t\in C([0,T];H)\) and
\(H\hookrightarrow X^*\) continuously,
\begin{equation*}
 v_t\in L^{p'}(0,T;W^{-1,p'}(\Omega)).
\end{equation*}
Together with \eqref{eq27}, this allows us to apply
Proposition~\ref{prop:product-rule} with \(w=u_t\) and \(z=v\).
Thus \(Y\) has an absolutely continuous representative,
which agrees with \(Y\) at every time because
\(u_t,v\in C([0,T];H)\).
Consequently,
\begin{equation}\label{eq28}
\begin{aligned}
 Y'(t)
 =\langle u_{tt}(t),v(t)\rangle
   +\langle v_t(t),u_t(t)\rangle=\langle u_{tt}(t),v(t)\rangle+\|u_t(t)\|_2^2
\end{aligned}
\end{equation}
for almost every \(t\in(0,T)\).
Here the second equality uses \(v_t=u_t\) and the
natural embedding \(H\hookrightarrow X^*\).

The weak equation, interpreted in \(X^*\), gives
\[
\begin{aligned}
 \langle u_{tt}(t),v(t)\rangle
 =-(\nabla u(t),\nabla v(t))
 -\int_\Omega|\nabla u_t(t)|^{p-2}
       \nabla u_t(t)\cdot\nabla v(t)\dd x
 +\int_\Omega|u(t)|^{r-1}u(t)v(t)\dd x
\end{aligned}
\]
for almost every \(t\in(0,T)\).
Combining this identity with \eqref{eq28} and
substituting \(v=u-\eta\) yields \eqref{eq26}.
\end{proof}

\begin{lemma}[Cumulative dissipation and damping estimate]
\label{lem:cumulative-dissipation}
Assume the hypotheses of Lemma~\ref{lem:shifted-multiplier}.
Choose
\[
 \mathcal R_0\geq\max\{1,\|\nabla v_0\|_p\},
\]
and define
\begin{equation}\label{eq29}
 \mathcal R(t)=\mathcal R_0+
       \int_0^t\|\nabla u_s(s)\|_p\dd s,
 \quad \forall t\in[0,T].
\end{equation}
Then \(\mathcal R\in W^{1,p}(0,T)\), and its absolutely
continuous representative is nondecreasing and satisfies
\begin{equation}\label{eq30}
 \mathcal R'(t)=\|\nabla u_t(t)\|_p,
 \quad\text{for a.e. }t\in(0,T).
\end{equation}
Furthermore,
\begin{equation}\label{eq31}
 \left|\int_\Omega|\nabla u_t(t)|^{p-2}
          \nabla u_t(t)\cdot\nabla v(t)\dd x\right|
 \leq \mathcal R(t)(\mathcal R'(t))^{p-1}
\end{equation}
for almost every \(t\in(0,T)\).
\end{lemma}

\begin{proof}
H\"older's inequality gives
\[
 \int_0^t\|\nabla u_s(s)\|_p\dd s
 \leq t^{(p-1)/p}
 \left(\int_0^t\|\nabla u_s(s)\|_p^p\dd s\right)^{1/p},
 \quad\forall t\in[0,T].
\]
Since \(u_t\in L^p(0,T;W_0^{1,p}(\Omega))\), the
right-hand side is finite and uniformly bounded on
\([0,T]\). Hence \(\mathcal R\) is bounded on \([0,T]\).
Moreover,
\(t\mapsto\|\nabla u_t(t)\|_p\) belongs to
\(L^p(0,T)\subset L^1(0,T)\).
Thus, by \eqref{eq29} and the fundamental theorem of calculus
for Lebesgue integrals, \(\mathcal R\) is absolutely continuous
on \([0,T]\), with \eqref{eq30} holding for almost
every \(t\in(0,T)\).
Since \(\mathcal R\) is bounded and
\(\mathcal R'\in L^p(0,T)\), it follows that
\(\mathcal R\in W^{1,p}(0,T)\).

Moreover, the nonnegativity of the integrand in \eqref{eq29}
implies that \(\mathcal R\) is nondecreasing and
\begin{equation}\label{eq32}
 \mathcal R(t)\geq\mathcal R_0\geq1,
 \quad\forall t\in[0,T].
\end{equation}
The energy identity \eqref{eq3} further yields
\[
 (\mathcal R'(t))^p=-E'(t),
 \quad\text{a.e. }t\in(0,T),
\]
and
\[
 \mathcal R(t)\leq\mathcal R_0+
 t^{(p-1)/p}\bigl(E(0)-E(t)\bigr)^{1/p},
 \quad\forall t\in[0,T].
\]
Using the integral representation of \(v\) in
\eqref{eq24}, we obtain
\begin{equation}\label{eq33}
 \|\nabla v(t)\|_p
 \leq\|\nabla v_0\|_p
      +\int_0^t\|\nabla u_s(s)\|_p\dd s
 \leq\mathcal R(t),
 \quad\forall t\in[0,T].
\end{equation}

For the nonlinear damping term, H\"older's inequality,
\eqref{eq30}, and \eqref{eq33} give
\[
\begin{aligned}
 &\left|\int_\Omega|\nabla u_t(t)|^{p-2}
          \nabla u_t(t)\cdot\nabla v(t)\dd x\right|\\
 &\quad\leq
 \left(\int_\Omega|\nabla u_t(t)|^{(p-1)p'}\dd x
 \right)^{1/p'}\|\nabla v(t)\|_p\\
 &\quad=\|\nabla u_t(t)\|_p^{p-1}\|\nabla v(t)\|_p
 \leq(\mathcal R'(t))^{p-1}\mathcal R(t),
 \quad\text{a.e. }t\in(0,T).
\end{aligned}
\]
This proves \eqref{eq31}.
\end{proof}

Since \(0<T<\Tmax\) was arbitrary, the preceding estimates
hold on every compact subinterval of \([0,\Tmax)\).
In the following lemmas and theorem proofs, all scalar functions
are considered on \([0,\Tmax)\).

\begin{lemma}\label{lem:functional}
Let \(\mathcal{R}(t)\) and \(Y(t)\) be defined by \eqref{eq29} and \eqref{eq25}, respectively.
Suppose \(h\) is locally absolutely continuous on
\([0,\Tmax)\) and satisfies, for almost everywhere $t\in(0,\Tmax)$,
\begin{equation}\label{eq34}
\begin{aligned}
 h'(t)=(\mathcal R'(t))^p,\quad
 Y'(t)\geq\nu h(t)+c_0\|u_t(t)\|_2^2
                    -\mathcal R(t)(\mathcal R'(t))^{p-1},
\end{aligned}
\end{equation}
where \(p<\nu\leq r+1\) and \(c_0>0\).
Choose
\begin{equation}\label{eq35}
 p<\lambda<\mu<\nu,\quad
 (3-p)\lambda<2,\quad
 \beta=\frac{\lambda-p}{p-1}.
\end{equation}
For \(a>0\), define the weighted auxiliary functional
\begin{equation}\label{eq36}
G(t)=h(t)+a\mathcal{R}^\beta(t) Y(t)-\frac{a^{p-1}}{\lambda}\mathcal{R}^\lambda(t), \quad \forall t\in[0,\Tmax).
\end{equation}
There are positive constants \(m\) and \(\widehat{C}\), depending only
on \(p,\lambda,\mu,c_0,\Omega\), such that, if
\begin{equation}\label{eq37}
 a^{3-p}\mathcal{R}_0^{\,2\beta+2-\lambda}\leq\frac{m}{2\widehat{C}},
\end{equation}
then
\begin{equation}\label{eq38}
\begin{aligned}
G'(t)\geq{}&
 a\mu \mathcal{R}^\beta(t) G(t)
 +a(\nu-\mu)\mathcal{R}^\beta(t) h(t)
 +\frac{ac_0}{2}\mathcal{R}^\beta(t)\|u_t(t)\|_2^2
 \\
 &+\frac12a^p\mathcal{R}^{\beta+\lambda}(t)
       \Phi\left(\frac{\mathcal{R}'(t)}{a\mathcal{R}^{\beta+1}(t)}\right),\quad \text{a.e. }t\in(0,\Tmax),
\end{aligned}
\end{equation}
where
\begin{equation*}
\Phi(z)=(z-1)(z^{p-1}-1)+\frac{\mu}{\lambda}-1,
 \quad z\geq0.
\end{equation*}
The constants \(m\) and \(\widehat{C}\) are given in \eqref{eq42} and \eqref{eq45}, respectively.
\end{lemma}

\begin{proof}
The parameters in \eqref{eq35} exist by the
argument in \eqref{eq13}, with \(\nu\) in place
of \(r+1\).  
Since \(h\), \(\mathcal{R}\), and \(Y\) are locally absolutely
continuous and \(\mathcal{R}(t)\geq1\) by \eqref{eq32}, the function \(G\) defined
in \eqref{eq36} is locally absolutely continuous. 
Differentiating \eqref{eq36} and using
\eqref{eq34} gives
\begin{equation}\label{eq39}
\begin{aligned}
 G'(t)&=h'(t)+a\mathcal{R}^\beta(t) Y'(t)
+a\beta \mathcal{R}^{\beta-1}(t)\mathcal{R}'(t)Y(t)
-a^{p-1}\mathcal{R}^{\lambda-1}(t)\mathcal{R}'(t)\\
&\geq
 (\mathcal{R}'(t))^p+a\nu \mathcal{R}^\beta(t) h(t)+ac_0\mathcal{R}^\beta(t)\|u_t(t)\|_2^2
 -a\mathcal{R}^{\beta+1}(t)(\mathcal{R}'(t))^{p-1}\\
 &\quad+a\beta \mathcal{R}^{\beta-1}(t)\mathcal{R}'(t)Y(t)
 -a^{p-1}\mathcal{R}^{\lambda-1}(t)\mathcal{R}'(t),\quad \text{a.e. }t\in(0,\Tmax).
\end{aligned}
\end{equation}
Using \eqref{eq36} again, one obtains
\begin{equation*}
\begin{aligned}
a\nu \mathcal{R}^\beta(t) h(t)&=a\mu \mathcal{R}^\beta(t) h(t)+a(\nu-\mu)\mathcal{R}^\beta(t) h(t)\\
&=a\mu \mathcal{R}^\beta(t) G(t)-a^2\mu \mathcal{R}^{2\beta}(t)Y(t)
      +\frac{\mu}{\lambda}a^p\mathcal{R}^{\beta+\lambda}(t)\\
&\quad+a(\nu-\mu)\mathcal{R}^\beta(t) h(t),\quad \forall t\in[0,\Tmax).
\end{aligned}
\end{equation*}
Substituting this expression into \eqref{eq39} yields
\begin{equation}\label{eq40}
\begin{aligned}
 G'(t)&\geq
 a\mu \mathcal{R}^\beta(t) G(t)+a(\nu-\mu)\mathcal{R}^\beta(t) h(t)
 +ac_0\mathcal{R}^\beta(t)\|u_t(t)\|_2^2\\
 &\quad+(\mathcal{R}'(t))^p-a\mathcal{R}^{\beta+1}(t)(\mathcal{R}'(t))^{p-1}
       -a^{p-1}\mathcal{R}^{\lambda-1}(t)\mathcal{R}'(t)
       +\frac{\mu}{\lambda}a^p\mathcal{R}^{\beta+\lambda}(t)
\\
 &\quad+\bigl(a\beta \mathcal{R}^{\beta-1}(t)\mathcal{R}'(t)
                    -a^2\mu \mathcal{R}^{2\beta}(t)\bigr)Y(t),\quad \text{a.e. }t\in(0,\Tmax).
\end{aligned}
\end{equation}

By \eqref{eq30}, put
\[
z=\frac{\mathcal{R}'(t)}{a\mathcal{R}^{\beta+1}(t)}\geq0, \quad \text{a.e. }t\in(0,\Tmax).
\]
Then using $
p(\beta+1)=\beta+\lambda$
by \eqref{eq35} 
, one obtains 
\begin{equation}\label{eq41}
\begin{aligned}
 &(\mathcal{R}'(t))^p-a\mathcal{R}^{\beta+1}(t)(\mathcal{R}'(t))^{p-1}
       -a^{p-1}\mathcal{R}^{\lambda-1}(t)\mathcal{R}'(t)
       +\frac{\mu}{\lambda}a^p\mathcal{R}^{\beta+\lambda}(t)\\
     &\quad=a^p\mathcal{R}^{\beta+\lambda}(t)
       \left(z^p-z^{p-1}-z+\frac{\mu}{\lambda}\right)
 \\&\quad=a^p\mathcal{R}^{\beta+\lambda}(t)\Phi(z),\quad \text{a.e. }t\in(0,\Tmax).
\end{aligned}
\end{equation}
For \(z\geq0\), the factors \(z-1\) and \(z^{p-1}-1\)
have the same sign. Consequently,
\[
 \Phi(z)=(z-1)(z^{p-1}-1)+\frac{\mu}{\lambda}-1\geq\frac{\mu}{\lambda}-1>0.
\]
For \(0\leq z\leq2\), positivity gives
\[
 \frac{\Phi(z)}{1+z^2}\geq\frac{\mu/\lambda-1}{5}.
\]
For \(z\geq2\), we have \(z-1\geq z/2\),
\(z^{p-1}-1\geq(1-2^{1-p})z^{p-1}\), and
\(1+z^2\leq5z^2/4\). Since \(p>2\), it follows that
\[
 \frac{\Phi(z)}{1+z^2}
 \geq\frac25(1-2^{1-p})z^{p-2}
 \geq\frac{2^{p-1}-1}{5}.
\]
Consequently,
\begin{equation}\label{eq42}
 m:=\inf_{z\geq0}\frac{\Phi(z)}{1+z^2}
 \geq\frac15\min\left\{\frac\mu\lambda-1,\ 2^{p-1}-1\right\}>0.
\end{equation}

By \eqref{eq25}, the Cauchy--Schwarz inequality,
\eqref{eq7}  and \eqref{eq33},
\begin{equation}\label{eq43}
\begin{aligned}
|Y(t)|&\leq\|v(t)\|_2\|u_t(t)\|_2
 \leq C_P\|\nabla v(t)\|_p\|u_t(t)\|_2\\
 &\leq C_P\mathcal{R}(t)\|u_t(t)\|_2,\quad \forall t\in[0,\Tmax).
\end{aligned}
\end{equation}
It follows from \eqref{eq43} and Young's inequality that 
\begin{equation}\label{eq44}
\begin{aligned}
 &\bigl|\bigl(a\beta \mathcal{R}^{\beta-1}(t)\mathcal{R}'(t)
                    -a^2\mu \mathcal{R}^{2\beta}(t)\bigr)Y(t)\bigr|\\
 &\leq a\beta C_P\mathcal{R}^\beta(t) \mathcal{R}'(t)\|u_t(t)\|_2+a^2\mu C_P\mathcal{R}^{2\beta+1}(t)\|u_t(t)\|_2\\
 &\leq\frac{ac_0}{2}\mathcal{R}^\beta(t)\|u_t(t)\|_2^2
       +\frac{a\beta^2C_P^2}{c_0}\mathcal{R}^\beta(t)(\mathcal{R}'(t))^2
       +\frac{a^3\mu^2C_P^2}{c_0}\mathcal{R}^{3\beta+2}(t)\\
       &\leq\frac{ac_0}{2}\mathcal{R}^\beta(t)\|u_t(t)\|_2^2
       + \widehat{C}a^3\mathcal{R}^{3\beta+2}(t)(1+z^2),\quad \text{a.e. }t\in(0,\Tmax),
\end{aligned}
\end{equation}
with 
\begin{equation}
\label{eq45}
\widehat{C}=\frac{C_P^2}{c_0}(\beta^2+\mu^2).
\end{equation}
Here we used \(\mathcal{R}'(t)=a\mathcal{R}^{\beta+1}(t)z$ for almost everywhere $t\in(0,\Tmax)$ in the last inequality.

Recall \(\mathcal{R}(t)\geq \mathcal{R}_0\) in \eqref{eq32} and 
\begin{equation}
\label{eq46}
2\beta+2-\lambda=\frac{(3-p)\lambda-2}{p-1}<0
\end{equation}
by \eqref{eq35}.
In particular, the factor
\(a^{3-p}\mathcal R(t)^{2\beta+2-\lambda}\) is bounded above by its
value at \(\mathcal R_0\).  
Thus \eqref{eq37} and \eqref{eq42} imply
\begin{equation*}
\begin{aligned}
\widehat{C}a^3\mathcal{R}^{3\beta+2}(t)(1+z^2)
 &=\widehat{C}a^{3-p}\mathcal{R}^{2\beta+2-\lambda}(t)
                  a^p\mathcal{R}^{\beta+\lambda}(t)(1+z^2)\\
 &\leq\frac{m}{2}a^p\mathcal{R}^{\beta+\lambda}(t)(1+z^2)\\
 &\leq\frac12a^p\mathcal{R}^{\beta+\lambda}(t)\Phi(z),\quad \text{a.e. }t\in(0,\Tmax).
\end{aligned}
\end{equation*}
Substituting this expression into \eqref{eq44}, we get 
\begin{equation}\label{eq47}
\begin{aligned}
 &\bigl|\bigl(a\beta \mathcal{R}^{\beta-1}(t)\mathcal{R}'(t)
                    -a^2\mu \mathcal{R}^{2\beta}(t)\bigr)Y(t)\bigr|\\
 &\quad\leq\frac{ac_0}{2}\mathcal{R}^\beta(t)\|u_t(t)\|_2^2
       + \frac12a^p\mathcal{R}^{\beta+\lambda}(t)\Phi(z),\quad \text{a.e. }t\in(0,\Tmax).
\end{aligned}
\end{equation}
Combining \eqref{eq41} and \eqref{eq47} with \eqref{eq40}
proves \eqref{eq38}.
\end{proof}

\begin{lemma}\label{lem:positive-h}
Assume the hypotheses of Lemma~\ref{lem:functional}.
Suppose in addition that
\begin{equation}\label{eq48}
h(t)\geq h(0)>0, \quad \forall t\in[0,\Tmax),
\end{equation}
and, for some constant \(M>0\),
\begin{equation}\label{eq49}
h(t)\leq M\mathcal{R}^{r+1}(t),
 \quad
 \|u_t(t)\|_2^2\leq M\mathcal{R}^{r+1}(t), \quad \forall t\in[0,\Tmax).
\end{equation}
Then \(\mathcal{R}_0\) and \(a\) can be chosen so that \(G(0)>0\)
and \eqref{eq37} holds. For these choices,
\begin{equation}\label{eq50}
G'(t)\geq a\mu C_*^{-\beta/(r+1)}
                    G^{1+\beta/(r+1)}(t), \quad \text{a.e. }t\in(0,\Tmax),
\end{equation}
and
\begin{equation}\label{eq51}
 \Tmax\leq
 \frac{r+1}{a\mu\beta}
 \left(\frac{C_*}{G(0)}\right)^{\beta/(r+1)}
 <\infty.
\end{equation}
Here $C_*=M+aC_P\sqrt M.$ 
\end{lemma}

\begin{proof}
First choose \(\kappa>0\) so small that
\begin{equation}\label{eq52}
 \frac{\kappa^{p-1}}{\lambda}\leq\frac{h(0)}4.
\end{equation}
The initial quantities \(h(0)\) and \(Y(0)\) are fixed before
choosing \(\mathcal{R}_0\).
For sufficiently large \(\mathcal{R}_0\), choose 
\begin{equation*}
a=\kappa \mathcal{R}_0^{-\lambda/(p-1)}.
\end{equation*}
Then by \eqref{eq36} and \eqref{eq35}, 
\begin{equation*}
G(0)=h(0)+\kappa \mathcal{R}_0^{-p/(p-1)}Y(0)
                    -\frac{\kappa^{p-1}}{\lambda}.
\end{equation*}
Choose \(\mathcal{R}_0\) large enough that
\[
 \kappa \mathcal{R}_0^{-p/(p-1)}|Y(0)|\leq\frac{h(0)}4.
\]
Together with \eqref{eq52}, this results in
\begin{equation*}
G(0)\geq\frac{h(0)}2>0.
\end{equation*}
At the same time, \eqref{eq46} gives
\begin{align*}
a^{3-p}\mathcal{R}_0^{2\beta+2-\lambda}
 &=\kappa^{3-p}
 \mathcal{R}_0^{-\lambda(3-p)/(p-1)+2\beta+2-\lambda}
 \\
 &=\kappa^{3-p}\mathcal{R}_0^{-2/(p-1)}
 \to0
 \quad\text{as }\mathcal{R}_0\to \infty.
\end{align*}
Thus \(\mathcal{R}_0\) can also be chosen to satisfy
\eqref{eq37} and
\(\mathcal{R}_0\geq\max\{1,\|\nabla v_0\|_p\}\).

By \eqref{eq48}, all terms after
\(a\mu \mathcal{R}^\beta(t) G(t)\) in \eqref{eq38} are nonnegative.
Hence
\begin{equation}\label{eq53}
G'(t)\geq a\mu \mathcal{R}^\beta(t) G(t), \quad \text{a.e. }t\in(0,\Tmax).
\end{equation}
Multiplying \eqref{eq53} by the integrating
factor \(e^{-a\mu\int_0^t\mathcal R^\beta(s)\dd s}\), we obtain
\[
 \frac{\dd}{\dd t}
 \left[e^{-a\mu\int_0^t\mathcal R^\beta(s)\dd s}G(t)\right]
 \geq0,
 \quad\text{a.e. }t\in(0,\Tmax).
\]
Integrating and using \(G(0)>0\), we conclude that
\[
 G(t)\geq
 G(0)e^{a\mu\int_0^t\mathcal R^\beta(s)\dd s}>0,
 \quad \forall t\in[0,\Tmax).
\]

Next, by \eqref{eq43} and \eqref{eq49},
\[
|Y(t)|\leq C_P\sqrt M\,\mathcal{R}^{1+(r+1)/2}(t), \quad \forall t\in[0,\Tmax).
\]
Because \(2\beta+2<\lambda<r+1\) by \eqref{eq46}, one has 
\begin{equation}\label{eq54}
 \beta+1+\frac{r+1}{2}<r+1.
\end{equation}
Since \(\mathcal{R}(t)\geq1\),
dropping the negative term in \eqref{eq36}
and using \eqref{eq54} and \eqref{eq49} yields
\begin{equation}\label{eq55}
\begin{aligned}
0<G(t)
 &\leq h(t)+a\mathcal{R}^\beta(t)|Y(t)|\\
 &\leq M\mathcal{R}^{r+1}(t)
       +aC_P\sqrt M\,\mathcal{R}^{\beta+1+(r+1)/2}(t)\\
 &\leq C_*\mathcal{R}^{r+1}(t),\quad \forall t\in[0,\Tmax).
\end{aligned}
\end{equation}
It follows that
\[
\mathcal{R}^\beta(t)
\geq\left(\frac{G(t)}{C_*}\right)^{\beta/(r+1)}
=C_*^{-\beta/(r+1)}G^{\beta/(r+1)}(t),\quad \forall t\in[0,\Tmax).
\]
Substitution into \eqref{eq53} proves
\eqref{eq50}.

Differentiating the negative power and using \eqref{eq50}, one obtains 
\[
\begin{aligned}
 \frac{\dd}{\dd t}G^{-\beta/(r+1)}(t)
 &\leq-\frac{\beta}{r+1}
   G(t)^{-1-\beta/(r+1)}
   \left[
   a\mu C_*^{-\beta/(r+1)}
   G^{1+\beta/(r+1)}(t)
   \right]\\
 &=-\frac{\beta}{r+1}
   a\mu C_*^{-\beta/(r+1)}, \quad \text{a.e. }t\in(0,\Tmax).
\end{aligned}
\]
Integrating, we obtain
\begin{equation*}
 0<G(t)^{-\beta/(r+1)}
 \leq G(0)^{-\beta/(r+1)}
 -\frac{a\mu\beta}{r+1}C_*^{-\beta/(r+1)}t,
 \quad\forall t\in[0,\Tmax).
\end{equation*}
Consequently,
\[
t<\frac{r+1}{a\mu\beta}
 \left(\frac{C_*}{G(0)}\right)^{\beta/(r+1)}, \quad\forall t\in[0,\Tmax).
\]
Taking the supremum over \(t\in[0,\Tmax)\) yields  \eqref{eq51}.
\end{proof}

With these lemmas in hand, we now turn to the proofs of the main theorems.
\subsection{Proof of Theorem~\ref{thm:negative}}
Choose \(\varepsilon\in(0,1)\) and \(\delta>0\) so that
\begin{equation}\label{eq56}
 \nu=(1-\varepsilon)(r+1)>p,\quad
 \frac{\nu}{2}-1-\delta>0.
\end{equation}
This is possible because \(2<p<r+1\).

We specify the error estimate in \eqref{eq26}.
For the elastic cross term, Young's inequality gives
\begin{equation}\label{eq57}
|(\nabla u(t),\nabla\eta)|
 \leq\delta\|\nabla u(t)\|_2^2
             +\frac1{4\delta}\|\nabla\eta\|_2^2, \quad \forall t\in[0,\Tmax).
\end{equation}
By H\"older's inequality and Young's inequality with
conjugate exponents \((r+1)/r\) and \(r+1\), we obtain
\begin{equation}\label{eq58}
\begin{aligned}
 \left|\int_\Omega|u(t)|^{r-1}u(t)\eta\dd x\right|
 &\leq\|u(t)\|_{r+1}^r\|\eta\|_{r+1}\\
 &\leq\varepsilon\|u(t)\|_{r+1}^{r+1}
 +\frac{r^r}{(r+1)^{r+1}\varepsilon^r}\|\eta\|_{r+1}^{r+1},\quad \forall t\in[0,\Tmax).
\end{aligned}
\end{equation}
For this fixed choice of \(\varepsilon,\delta\), write
\begin{equation}\label{eq59}
 C_\eta:=
 \frac1{4\delta}\|\nabla\eta\|_2^2
 +\frac{r^r}{(r+1)^{r+1}\varepsilon^r}
                            \|\eta\|_{r+1}^{r+1}.
\end{equation}

Substituting \eqref{eq57},
\eqref{eq58}, and \eqref{eq31}
into \eqref{eq26}, we get
\begin{equation}\label{eq60}
\begin{aligned}
 Y'(t)\geq{}&\|u_t(t)\|_2^2-(1+\delta)\|\nabla u(t)\|_2^2
                 +(1-\varepsilon)\|u(t)\|_{r+1}^{r+1}\\
 &-C_\eta-\mathcal R(t)(\mathcal R'(t))^{p-1},\quad \text{a.e. }t\in(0,\Tmax).
\end{aligned}
\end{equation}
Using the definition of \(E(t)\) and
\(\nu=(1-\varepsilon)(r+1)\) in
\eqref{eq56}, we rewrite the source term as
\[
 (1-\varepsilon)\|u(t)\|_{r+1}^{r+1}
 =-\nu E(t)+\frac{\nu}{2}\|u_t(t)\|_2^2
          +\frac{\nu}{2}\|\nabla u(t)\|_2^2.
\]
Substituting this identity into \eqref{eq60}
yields
\begin{equation}\label{eq61}
\begin{aligned}
Y'(t)\geq{}&
 -\nu E(t)+\left(1+\frac{\nu}{2}\right)\|u_t(t)\|_2^2
 +\left(\frac{\nu}{2}-1-\delta\right)\|\nabla u(t)\|_2^2
 \\
 &-C_\eta-\mathcal{R}(t)(\mathcal{R}'(t))^{p-1},\quad \text{a.e. }t\in(0,\Tmax).
\end{aligned}
\end{equation}

By the density of
\(C_c^\infty(\Omega)\) in \(H_0^1(\Omega)\), we may choose
\(v_0\in C_c^\infty(\Omega)\) so that, with
\(\eta=u_0-v_0\), the norm \(\|\nabla\eta\|_2\) is
arbitrarily small. The Sobolev inequality
\eqref{eq10} then gives
\[
 \|\eta\|_{r+1}\leq R^{1/(r+1)}\|\nabla\eta\|_2,
\]
and hence
\[
 C_\eta
 \leq \frac{1}{4\delta}\|\nabla\eta\|_2^2
 +\frac{r^r R}
 {(r+1)^{r+1}\varepsilon^r}
 \|\nabla\eta\|_2^{r+1}
 \to0
 \quad\text{as }v_0\to u_0\text{ in }H_0^1(\Omega).
\]
Since \(E(0)<0\), we can therefore fix such an
approximation \(v_0\) satisfying
\begin{equation}
\label{eq62}
C_\eta<\frac{\nu}{2}(-E(0)).
\end{equation}

Define
\begin{equation*}
h(t)=-E(t)-\frac{C_\eta}{\nu}, \quad \forall t\in[0,\Tmax).
\end{equation*}
The energy identity \eqref{eq3} and \eqref{eq62} imply
\[
 h(t)\geq h(0)>\frac{-E(0)}{2}>0,
 \quad\forall t\in[0,\Tmax),
\]
and
\begin{equation}
\label{eq63}
h'(t)=\|\nabla u_t(t)\|_p^p=(\mathcal{R}'(t))^p, \quad\text{a.e. }t\in(0,\Tmax).
\end{equation}
Dropping the positive gradient term in
\eqref{eq61} yields
\[
Y'(t)\geq\nu h(t)+\left(1+\frac{\nu}{2}\right)\|u_t(t)\|_2^2
                      -\mathcal{R}(t)(\mathcal{R}'(t))^{p-1},\quad \text{a.e. }t\in(0,\Tmax).
\]
Hence \eqref{eq34} holds with
\(c_0=1+\nu/2\).

We next verify the upper bounds required by
Lemma~\ref{lem:positive-h}.
Using \(u(t)=v(t)+\eta,\ \mathcal{R}(t)\geq1\) for all $t\in[0,\Tmax)$, together with
\eqref{eq7} and \eqref{eq33}, we obtain
\begin{equation}\label{eq64}
\|u(t)\|_{r+1}
 \leq C_E \mathcal{R}(t)+\|\eta\|_{r+1}
 \leq(C_E+\|\eta\|_{r+1})\mathcal{R}(t), \quad \forall t\in[0,\Tmax).
\end{equation}
Since \(E(t)<0\), we have 
\[
h(t)\leq-E(t)\leq\frac1{r+1}\|u(t)\|_{r+1}^{r+1}, \quad \forall t\in[0,\Tmax),
\]
and
\[
\begin{aligned}
 \|u_t(t)\|_2^2
 &=2E(t)-\|\nabla u(t)\|_2^2
       +\frac2{r+1}\|u(t)\|_{r+1}^{r+1}\\
 &\leq\frac2{r+1}\|u(t)\|_{r+1}^{r+1},\quad \forall t\in[0,\Tmax).
\end{aligned}
\]
Together with \eqref{eq64}, these inequalities result in \eqref{eq49} with the explicit constant
\[
 M=\frac{2}{r+1}(C_E+\|\eta\|_{r+1})^{r+1}.
\]

We now verify the explicit parameter choices
\eqref{eq14}--\eqref{eq15}
when \(h(0)=h_0>0\), \(Y(0)=(v_0,u_1)\), and
\(c_0=1+\nu/2\). By the definition of \(\kappa\),
\[
 \frac{\kappa^{p-1}}{\lambda}=\frac{h_0}{4}.
\]
Moreover, the choice of \(\mathcal{R}_0\) in \eqref{eq15} ensures that
\[
 \mathcal{R}_0^{p/(p-1)}
 \geq\frac{4\kappa|(v_0,u_1)|}{h_0},
 \quad
 \mathcal{R}_0^{2/(p-1)}
 \geq\frac{2\widehat{C}}{m}\kappa^{3-p}.
\]
Consequently,
\[
 \kappa \mathcal{R}_0^{-p/(p-1)}|(v_0,u_1)|
 \leq\frac{h_0}{4},
\]
and, using \(a=\kappa \mathcal{R}_0^{-\lambda/(p-1)}\),
\[
 a^{3-p}\mathcal{R}_0^{2\beta+2-\lambda}
 =\kappa^{3-p}\mathcal{R}_0^{-2/(p-1)}
 \leq\frac{m}{2\widehat{C}}.
\]
Thus the condition \eqref{eq37}
is satisfied. Furthermore,
\[
 a\mathcal{R}_0^\beta=\kappa \mathcal{R}_0^{-p/(p-1)},
 \quad
 \frac{a^{p-1}}{\lambda}\mathcal{R}_0^\lambda=\frac{h_0}{4},
\]
so the initial value of \(G\) by \eqref{eq36} satisfies
\begin{equation}\label{eq65}
 \begin{aligned}
 G(0)
 &=\frac{3h_0}{4}
   +\kappa \mathcal{R}_0^{-p/(p-1)}(v_0,u_1)\\
 &\geq\frac{3h_0}{4}
   -\kappa \mathcal{R}_0^{-p/(p-1)}|(v_0,u_1)|
 \geq\frac{h_0}{2}>0.
 \end{aligned}
\end{equation}

The preceding estimates verify all hypotheses of
Lemma~\ref{lem:positive-h} on every compact subinterval
of \([0,\Tmax)\), with the same constant
$
 M=\frac{2}{r+1}(C_E+\|\eta\|_{r+1})^{r+1}.
$
Applying the lifespan bound \eqref{eq51} and \eqref{eq65}, we obtain
\[
 \Tmax
 \leq\frac{r+1}{a\mu\beta}
 \left(\frac{M+aC_P\sqrt M}{G(0)}\right)^{\beta/(r+1)}
 \leq\frac{r+1}{a\mu\beta}
 \left(\frac{2(M+aC_P\sqrt M)}{h_0}\right)^{\beta/(r+1)}.
\]
This proves \eqref{eq16}. In particular,
\(\Tmax<\infty\), and the continuation alternative
in Corollary~\ref{cor:continuation} yields
\[
 \limsup_{t\uparrow\Tmax}
 \left(\|u_t(t)\|_2^2+\|\nabla u(t)\|_2^2\right)=\infty.
\]

\subsection{Potential well theory}\label{sec3.2}

We first justify \eqref{eq8} and \eqref{eq12}.
For any nonzero \(w\in H_0^1(\Omega)\), differentiation shows
\[
 \frac{\mathrm d}{\mathrm d\tau}J(\tau w)
 =\tau\|\nabla w\|_2^2-\tau^r\|w\|_{r+1}^{r+1}.
\]
Since \(r>1\), the function \(\tau\mapsto J(\tau w)\)
attains its maximum on \([0,\infty)\) at
\[
 \tau_w=
 \left(\frac{\|\nabla w\|_2^2}
            {\|w\|_{r+1}^{r+1}}\right)^{1/(r-1)}.
\]
Moreover, \(\tau_w w\in\mathcal N\), and every element of
\(\mathcal N\) is obtained in this way, with \(\tau_w=1\)
when \(w\in\mathcal N\). This proves the equality of the
variational definitions in \eqref{eq8}. At the maximum,
\begin{align*}
 J(\tau_w w)
 =\frac{r-1}{2(r+1)}\tau_w^2\|\nabla w\|_2^2=\frac{r-1}{2(r+1)}
 \left(\frac{\|\nabla w\|_2^{r+1}}
            {\|w\|_{r+1}^{r+1}}\right)^{2/(r-1)}.
\end{align*}
Taking the infimum over nonzero \(w\) and using
\eqref{eq9} gives
\[
 d=\frac{r-1}{2(r+1)}R^{-2/(r-1)}
   =\mathcal G(s_0)=\widetilde d.
\]
The formula also holds at \(r+1=6\), since the argument
does not require the Sobolev supremum to be attained.

\begin{lemma}[Invariance of \(\mathscr W_2\)]
\label{lem:well}
Suppose \(E(0)<d\) and \(u_0\in\mathscr W_2\). Then
\begin{equation}\label{eq66}
 u(t)\in\mathscr W_2,\quad\forall t\in[0,\Tmax).
\end{equation}
Moreover, there is a unique \(\tilde{s}>s_0\) satisfying
\begin{equation}\label{eq67}
 \mathcal G(\tilde{s})=\frac12\tilde{s}^2-\frac{R}{r+1}\tilde{s}^{r+1}=E(0),
\end{equation}
and
\begin{equation}\label{eq68}
 \|\nabla u(t)\|_2\geq \tilde{s}>s_0,
 \quad\forall t\in[0,\Tmax).
\end{equation}
\end{lemma}

\begin{proof}
The energy identity implies
\[
J(u(t))\leq E(t)\leq E(0)<d, \quad \forall t\in[0,\Tmax),
\]
so \(u(t)\in\mathscr W\) throughout the existence interval.

At any time \(t\in[0,\Tmax)\) for which
\(u(t)\in\mathscr W_2\), \eqref{eq9} implies 
\[
\begin{aligned}
0<\|\nabla u(t)\|_2^2<\|u(t)\|_{r+1}^{r+1}\leq R\|\nabla u(t)\|_2^{r+1}.
\end{aligned}
\]
Dividing by \(\|\nabla u(t)\|_2^2\) and recalling \eqref{eq11}, we obtain
\begin{equation}\label{eq69}
\|\nabla u(t)\|_2>R^{-1/(r-1)}=s_0.
\end{equation}
The maps
\(t\mapsto\|\nabla u(t)\|_2^2-\|u(t)\|_{r+1}^{r+1}\)
and \(t\mapsto\|\nabla u(t)\|_2\) are continuous.
If the first map reached zero for the first time at
\(t_0\in(0,\Tmax)\),
\eqref{eq69} and continuity would give
\(\|\nabla u(t_0)\|_2\geq s_0>0\). Hence
\(u(t_0)\in\mathcal N\), and the definition of \(d\) would imply
\[
E(t_0)=\frac12\|u_t(t_0)\|_2^2+J(u(t_0))\geq d.
\]
contrary to \(E(t_0)\leq E(0)<d\). Hence
\(\|\nabla u(t)\|_2^2<\|u(t)\|_{r+1}^{r+1}\) for all $t\in[0,\Tmax)$, proving \eqref{eq66}. In particular,
\eqref{eq69} now applies at every such time.

For the quantitative gap, differentiate \(\mathcal G\), then 
\[
 \mathcal G'(s)=s-Rs^r=s(1-Rs^{r-1}).
\]
Obviously, \(\mathcal G\) is strictly decreasing on
\((s_0,\infty)\), with \(\mathcal G(s_0)=d>E(0)\) and
\(\mathcal G(s)\to-\infty\) as \(s\to\infty\).
This proves existence and uniqueness of \(\tilde{s}\) in
\eqref{eq67}.

Finally,
\[
\mathcal G(\|\nabla u(t)\|_2)
 \leq J(u(t))\leq E(t)\leq E(0)=\mathcal G(\tilde{s}), \quad \forall t\in[0,\Tmax).
\]
By \eqref{eq69}, both arguments lie on
this decreasing branch, so \(\|\nabla u(t)\|_2\geq \tilde{s}\) for all $t\in[0,\Tmax)$. This proves \eqref{eq68}.
\end{proof}

\begin{proof}[Proof of Theorem~\ref{thm:well}]
Let \(E(0)<d\) and \(u_0\in\mathscr W_2\). Set
\begin{equation*}
h(t)=d-E(t), \quad \forall t\in[0,\Tmax).
\end{equation*}
Then \eqref{eq63} holds and 
\[
 h(t)\geq d-E(0)>0, \quad\forall t\in[0,\Tmax).
\]

We now give explicit upper bounds for \(\varepsilon\) and
\(\delta\) that ensure \eqref{eq17}. By Lemma~\ref{lem:well}, we have \(\tilde s>s_0\).
Together with \eqref{eq12}, this shows
\[
  \mathcal{A}=\frac{r-1}{2}\tilde s^2-(r+1)d
 =\frac{r-1}{2}(\tilde s^2-s_0^2)>0.
\]
Set
\[
\mathcal{B}=(r+1)\left(\frac{\tilde{s}^2}{2}-d\right).
\]
Using \((r+1)d=(r-1)s_0^2/2\), we obtain
\(\mathcal{B}= \mathcal{A}+\tilde{s}^2>0\).
We may therefore choose
\[
 0<\varepsilon<
 \min\left\{1-\frac{p}{r+1},\,
             \frac{\mathcal{A}}{4\mathcal{B}}\right\},
 \quad
 0<\delta<\frac{\mathcal{A}}{4\tilde{s}^2},
 \quad
 \nu=(1-\varepsilon)(r+1).
\]
The first restriction on \(\varepsilon\) ensures
\(\nu>p\). Moreover,
\[
\begin{aligned}
 \left(\frac{\nu}{2}-1-\delta\right)\tilde{s}^2-\nu d
 &=
 \frac{r-1}{2}\tilde{s}^2-(r+1)d
 -\varepsilon(r+1)\left(\frac{\tilde{s}^2}{2}-d\right)
 -\delta \tilde{s}^2\\
 &=\mathcal{A}-\varepsilon \mathcal{B}-\delta \tilde{s}^2>\frac{\mathcal{A}}{2}>0.
\end{aligned}
\]
Since $\nu d>0$, we also have $\frac{\nu}{2}-1-\delta>0.$

Choose \(v_0\in C_c^\infty(\Omega)\) approximating \(u_0\),
and define \(\eta,v,Y,\mathcal{R}\) as before. The estimate \eqref{eq60} remains valid,
with \(C_\eta\) defined by \eqref{eq59}.
By the definition of the energy and \(h(t)=d-E(t)\), we have
\[
 \|u(t)\|_{r+1}^{r+1}
 =(r+1)\left(
 \frac12\|u_t(t)\|_2^2
 +\frac12\|\nabla u(t)\|_2^2+h(t)-d
 \right),
 \quad t\in[0,\Tmax).
\]
Substituting this expression into \eqref{eq60}
and using \(\nu=(1-\varepsilon)(r+1)\), we obtain
\begin{equation}\label{eq70}
\begin{aligned}
 Y'(t)\geq{}&\nu h(t)+\left(1+\frac\nu2\right)\|u_t(t)\|_2^2
 +\left(\frac\nu2-1-\delta\right)\|\nabla u(t)\|_2^2\\
 &-\nu d-C_\eta-\mathcal R(t)(\mathcal R'(t))^{p-1},\quad \text{a.e. }t\in(0,\Tmax).
\end{aligned}
\end{equation}
The parameters \(\varepsilon,\delta\) have already been fixed.
Choose the approximation so accurately that
\[
 C_\eta<
 \left(\frac{\nu}{2}-1-\delta\right)\tilde{s}^2-\nu d.
\]
By \eqref{eq68}, the gradient term in
\eqref{eq70} absorbs both \(\nu d\) and \(C_\eta\).
Hence
\[
Y'(t)\geq\nu h(t)+\left(1+\frac\nu2\right)\|u_t(t)\|_2^2
                   -\mathcal R(t)(\mathcal R'(t))^{p-1}, \quad \text{a.e. }t\in(0,\Tmax).
\]
This proves \eqref{eq34} with
\(c_0=1+\nu/2\).

It remains to check \eqref{eq49}. We have
\[
h(t)=d-E(t)\leq d+\frac1{r+1}\|u(t)\|_{r+1}^{r+1}, \quad \forall t\in[0,\Tmax),
\]
and
\begin{equation*}
\begin{aligned}
 \|u_t(t)\|_2^2
 &=2E(t)-\|\nabla u(t)\|_2^2+\frac2{r+1}\|u(t)\|_{r+1}^{r+1}\\
 &\leq2\max\{E(0),0\}+\frac2{r+1}\|u(t)\|_{r+1}^{r+1},\quad \forall t\in[0,\Tmax).
\end{aligned}
\end{equation*}
 Inequality \eqref{eq64} and \(\mathcal{R}\geq1\) now give
\eqref{eq49} with
\[
 M=d+2\max\{E(0),0\}
       +\frac{2}{r+1}(C_E+\|\eta\|_{r+1})^{r+1}.
\]
With \(h_0=d-E(0)\), the parameter choices in
Theorem~\ref{thm:well} and \eqref{eq65} show
\(G(0)\geq h_0/2=(d-E(0))/2\). Applying \eqref{eq51} from Lemma~\ref{lem:positive-h}
then yields \eqref{eq18}. 
In particular, \(\Tmax<\infty\), and
Corollary~\ref{cor:continuation} yields the asserted blow-up
of the natural energy norm.
\end{proof}

\begin{proof}[Proof of Corollary~\ref{cor:nonpositive}]
If \(E(0)\leq0\), then
\[
 \|u_0\|_{r+1}^{r+1}
 \geq\frac{r+1}{2}
        \bigl(\|u_1\|_2^2+\|\nabla u_0\|_2^2\bigr).
\]
Therefore
\[
 \|\nabla u_0\|_2^2-\|u_0\|_{r+1}^{r+1}\leq
 -\frac{r+1}{2}\|u_1\|_2^2
 -\frac{r-1}{2}\|\nabla u_0\|_2^2<0
\]
for every nonzero initial state. Moreover,
\(J(u_0)\leq E(0)\leq0<d\), so \(u_0\in\mathscr W_2\).
Theorem~\ref{thm:well} applies.
\end{proof}

\subsection{Proof of the positive-energy criterion}

\begin{proof}[Proof of Theorem~\ref{thm:positive}]
We use the same error estimates as in the negative-energy proof.  Estimate \eqref{eq61} still holds with \(C_\eta\)
defined by \eqref{eq59}. 
Let us also set 
\[
h(t)=-E(t)-\frac{C_\eta}{\nu}, \quad \forall t\in[0,\Tmax).
\]
Then \eqref{eq63} holds and 
\[
 h(t)\geq h(0)=-\big(E(0)+\frac{C_\eta}{\nu}\big),\quad\forall t\in[0,\Tmax).
\]
Since \(\nu/2-1-\delta>0\), the gradient term in \eqref{eq61} can be
discarded. With \(c_0=1+\nu/2\), this yields
\[
Y'(t)\geq\nu h(t)+c_0\|u_t(t)\|_2^2
                  -\mathcal R(t)(\mathcal R'(t))^{p-1},\quad \text{a.e. }t\in(0,\Tmax).
\]
Thus the comparison assumptions \eqref{eq34}
hold with the parameters specified in the theorem.

Use the constants \(m,\widehat C,\widetilde C\) fixed in
Theorem~\ref{thm:positive}.  By \eqref{eq22},
\[
 a^{p-1}\mathcal R_0^\lambda=2\lambda \big(E(0)+\frac{C_\eta}{\nu}\big).
\]
Using \eqref{eq46}, \eqref{eq20}, and
\eqref{eq19}, we obtain
\[
 \begin{aligned}
 a^{3-p}\mathcal R_0^{2\beta+2-\lambda}
 &=\kappa^{3-p}\mathcal R_0^{-2/(p-1)}\\
 &=(2\lambda \big(E(0)+\frac{C_\eta}{\nu}\big))^{(3-p)/(p-1)}
                          \mathcal R_0^{-2/(p-1)}
 \leq\frac{m}{2\widehat C}.
 \end{aligned}
\]
Consequently, condition \eqref{eq37}
holds, and Lemma~\ref{lem:functional} applies to
\[
G(t)=h(t)+a\mathcal R(t)^\beta Y(t)
                       -\frac{a^{p-1}}{\lambda}\mathcal R(t)^\lambda, \quad \forall t\in[0,\Tmax).
\]

Since \(h\) may be negative in the positive-energy case,
we must first show that the last term in \eqref{eq38}
compensates for the possibly negative contribution
\(a(\nu-\mu)\mathcal R(t)^\beta h(t)\).
Recall that \(\mu=(\lambda+\nu)/2\), so
\[
 \Phi(z)\geq\frac{\mu}{\lambda}-1
 =\frac{\nu-\mu}{\lambda}>0,
 \quad z\geq0.
\]
Using
\(h(t)\geq-\bigl(E(0)+C_\eta/\nu\bigr)\)
and setting
\(z=\mathcal R'(t)/(a\mathcal R(t)^{\beta+1})\),
we obtain
\[
\begin{aligned}
 &a(\nu-\mu)\mathcal R(t)^\beta h(t)
       +\frac12a^p\mathcal R(t)^{\beta+\lambda}\Phi(z)\\
 &\quad\geq
 -a(\nu-\mu)\left(E(0)+\frac{C_\eta}{\nu}\right)
       \mathcal R(t)^\beta
 +\frac{\nu-\mu}{2\lambda}
       a^p\mathcal R(t)^{\beta+\lambda}\\
 &\quad=
 a(\nu-\mu)\mathcal R(t)^\beta
 \left(
 \frac{a^{p-1}\mathcal R(t)^\lambda}{2\lambda}
 -\left(E(0)+\frac{C_\eta}{\nu}\right)
 \right)
 \geq0,\quad \text{a.e. }t\in(0,\Tmax).
\end{aligned}
\]
Here the final inequality follows from
\(\mathcal R(t)\geq\mathcal R_0\) and
$
 a^{p-1}\mathcal R_0^\lambda
 =2\lambda\left(E(0)+\frac{C_\eta}{\nu}\right).
$
Therefore, it follows from \eqref{eq38} that
\[
G'(t)\geq a\mu\mathcal R(t)^\beta G(t), \quad \text{a.e. }t\in(0,\Tmax).
\]
At the initial time, 
\eqref{eq22} yields
\[
 \begin{aligned}
 G(0)&=-\big(E(0)+\frac{C_\eta}{\nu}\big)+\kappa\mathcal R_0^{-p/(p-1)}(v_0,u_1)
                         -\frac{\kappa^{p-1}}{\lambda}\\
 &=\kappa\mathcal R_0^{-p/(p-1)}(v_0,u_1)-3\big(E(0)+\frac{C_\eta}{\nu}\big)>0,
 \end{aligned}
\]
where positivity is exactly \eqref{eq21}.
The integrating-factor argument used in \eqref{eq53}
therefore yields \[G(t)>0,\ \forall t\in[0,\Tmax).\]

To obtain the upper bound for \(G(t)\), recall \eqref{eq64}. 
Moreover, \(C_\eta\geq0\) and \(E(t)\leq E(0)\) implies
\[
\begin{aligned}
 h(t)&=-E(t)-\frac{C_\eta}{\nu}
       \leq\frac1{r+1}\|u(t)\|_{r+1}^{r+1},\\
 \|u_t(t)\|_2^2
 &=2E(t)-\|\nabla u(t)\|_2^2+\frac2{r+1}\|u(t)\|_{r+1}^{r+1}\\
 &\leq2E(0)+\frac2{r+1}\|u(t)\|_{r+1}^{r+1},\quad \forall t\in[0,\Tmax).
\end{aligned}
\]
Hence \eqref{eq49} holds with the constant
\[
 M=2E(0)+\frac{2}{r+1}(C_E+\|\eta\|_{r+1})^{r+1}
\]
specified in the theorem. The upper bound in \eqref{eq55} remains valid
without assuming \(h\geq0\).
Together with the positivity of \(G\) and the inequality
\(G'(t)\geq a\mu\mathcal R^\beta(t) G(t)\) established above,
the same argument leading to \eqref{eq51} gives
\[
 \Tmax\leq\frac{r+1}{a\mu\beta}
 \left(\frac{M+aC_P\sqrt M}{G(0)}\right)^{\beta/(r+1)}
 <\infty.
\]
Substituting the computed value of \(G(0)\) yields
\eqref{eq23}, and
Corollary~\ref{cor:continuation} gives \eqref{eq5}.\end{proof}

\begin{remark}[Nonemptiness of the positive-energy class]
\label{rem:positive-class}
For every prescribed positive energy level \(E(0)\), there
exist initial data satisfying the criterion in
Theorem~\ref{thm:positive}.
Fix admissible \(\varepsilon,\delta,\nu,\lambda\)
independently of the data below.
Choose a nonzero \(\varphi\in C_c^\infty(\Omega)\) and set
\[
 u_0=L\varphi,\quad u_1=b_L\varphi,
\]
where
\[
 b_L=
 \left(
 \frac{2E(0)-L^2\|\nabla\varphi\|_2^2
       +\dfrac{2}{r+1}L^{r+1}\|\varphi\|_{r+1}^{r+1}}
      {\|\varphi\|_2^2}
 \right)^{1/2}.
\]
Since \(r+1>2\), \(b_L\) is well defined and positive
for sufficiently large \(L\), and these data have
exactly the prescribed energy.

Take \(v_0=u_0\), so that \(C_\eta=0\), and choose
\[
 \mathcal R_0=
 \max\left\{
 1,L\|\nabla\varphi\|_p,
 \widetilde C E(0)^{(3-p)/2}
 \right\}.
\]
Then \eqref{eq20} holds, and
\(\mathcal R_0=L\|\nabla\varphi\|_p\) for sufficiently
large \(L\). Moreover, the definition of \(b_L\) and \(r+1>2\) imply
\[
 \lim_{L\to\infty}\frac{b_L}{L^{(r+1)/2}}
 =
 \left(
 \frac{2\|\varphi\|_{r+1}^{r+1}}
 {(r+1)\|\varphi\|_2^2}
 \right)^{1/2}>0.
\]
Since \(r+1>p>2\), we have
\[
 \frac{r+3}{2}>\frac{p+2}{2}>
 \frac{p}{p-1}.
\]
Consequently,
\[
 \lim_{L\to\infty}
 \frac{(v_0,u_1)}{\mathcal R_0^{p/(p-1)}}
 =
 \lim_{L\to\infty}
 \frac{Lb_L\|\varphi\|_2^2}
      {\mathcal R_0^{p/(p-1)}}
 =+\infty.
\]
As \(E(0)\) and the auxiliary parameters are fixed,
\eqref{eq21} follows for sufficiently large \(L\).
\end{remark}

\appendix
\section{Auxiliary results}

\begin{proposition}[{\cite[Proposition~3.2]{RW2012}}]
\label{prop:product-rule}
Let \(H\) be a Hilbert space and \(X\) a Banach space such that
\(X\hookrightarrow H\hookrightarrow X^*\) continuously and densely.
Let \(1<p<\infty\), \(p'=p/(p-1)\), and suppose that
\[
 w,z\in L^p(0,T;X)\cap L^2(0,T;H),\quad
 w',z'\in L^{p'}(0,T;X^*).
\]
Then \((w,z)_H\) has an absolutely continuous representative and
\[
 \frac{\dd}{\dd t}(w(t),z(t))_H
 =\langle w'(t),z(t)\rangle+\langle z'(t),w(t)\rangle, \quad\text{a.e. }t\in(0,T).
\]
If \(w,z\in C([0,T];H)\), their scalar product agrees with
this representative at every time.
\end{proposition}

\begin{lemma}[Uniform local existence time]
\label{lem:uniform-restart}
Let Assumption~\ref{ass:local} hold.
For every \(M>0\), there exists
\(\tau_M=\tau(M;f,p,\Omega)>0\) such that every
initial state
\((w_0,w_1)\in H_0^1(\Omega)\times L^2(\Omega)\)
satisfying
\begin{equation}\label{eq71}
 \|\nabla w_0\|_2^2+\|w_1\|_2^2\leq M
\end{equation}
admits a unique energy weak solution on
\([0,\tau_M]\).
\end{lemma}

\begin{proof}
We follow the local construction in
\cite[Lemma~2.2 and Sections~2.4--2.5]{RW2012}.
For each initial state satisfying
\eqref{eq71}, choose the approximating
initial data as in \cite[Section~2.4]{RW2012}, with
\[
 \|\nabla w_{0,m}\|_2^2+\|w_{1,m}\|_2^2
 \leq(\sqrt M+1)^2.
\]
Fix a common cutoff radius \(K\) such that
\[
 K^2>2(\sqrt M+1)^2.
\]

By \cite[Lemma~2.4]{RW2012}, the approximating sources
\(f_m\) have local Lipschitz constants as maps from
\(H_0^1(\Omega)\) into \(L^q(\Omega)\) that are uniform
in \(m\) on the ball of radius \(K\), where \(q\) is
defined in \cite[(1.4)]{RW2012}.
Moreover, \(f_m(0)=f(0)\).
Consequently, the constants in the estimates
\cite[(2.36)--(2.37)]{RW2012} can be chosen independently
of \(m\) and of the particular initial state satisfying
\eqref{eq71}.
Since the quadratic initial energies are bounded by
\((\sqrt M+1)^2/2\), these estimates provide a common
existence time
\(\tau_M=\tau(M;f,p,\Omega)>0\)
for all the approximating problems.

The convergence argument in
\cite[Sections~2.4--2.5]{RW2012} then yields a weak
solution on \([0,\tau_M]\) for each such initial state.
The energy identity and uniqueness follow from
\cite[Theorem~1.2]{RW2012}, since \(r\leq5\).
\end{proof}

\begin{proof}[Proof of Corollary~\ref{cor:continuation}]
By local existence and uniqueness, solutions with the
same initial data agree on their common interval.
They therefore define a unique maximal solution
on \([0,\Tmax)\), with the stated regularity and
energy identity on every compact subinterval.

Suppose that \(\Tmax<\infty\) and
\[
 \limsup_{t\uparrow\Tmax}\mathcal E(t)<\infty.
\]
Then there exist \(t_*\in[0,\Tmax)\) and \(M>0\) such that
\[
 \|\nabla u(t)\|_2^2+\|u_t(t)\|_2^2\leq M,
 \quad \forall t\in[t_*,\Tmax).
\]
Let \(\tau_M>0\) be given by
Lemma~\ref{lem:uniform-restart}.
Choose \(t_0\in[t_*,\Tmax)\) such that
\(\Tmax-t_0<\tau_M\).
The lemma provides an energy weak solution \(w\)
on \([0,\tau_M]\) with
\[
 w(0)=u(t_0),\quad w_t(0)=u_t(t_0).
\]

Since the equation is autonomous, the time-translated function
\(z(t):=w(t-t_0)\) satisfies
\[
 z_{tt}-\Delta z-\Delta_p z_t=f(z)
\]
on \([t_0,t_0+\tau_M]\), with the same homogeneous Dirichlet
boundary condition and
\[
 z(t_0)=u(t_0),\quad z_t(t_0)=u_t(t_0).
\]
By uniqueness,
\[
 z(t)=u(t),\quad z_t(t)=u_t(t),
 \quad \forall t\in[t_0,\Tmax).
\]
Define
\[
 \widetilde u(t)=
 \begin{cases}
  u(t),&0\leq t\leq t_0,\\
 z(t),&t_0<t\leq t_0+\tau_M.
 \end{cases}
\]
The two pieces have the required regularity, and
their displacement and velocity agree at \(t_0\).
Consequently,
\[
 \widetilde u\in
 C([0,t_0+\tau_M];H_0^1(\Omega)),
 \quad
 \widetilde u_t\in
 C([0,t_0+\tau_M];L^2(\Omega))
 \cap L^p(0,t_0+\tau_M;W_0^{1,p}(\Omega)).
\]
Adding the weak formulations on the two intervals
gives the weak formulation for \(\widetilde u\),
since the endpoint terms at \(t_0\) cancel.
Likewise, adding the energy identities yields the
energy identity on the combined interval.
Hence \(\widetilde u\) is an energy weak solution
that agrees with \(u\) on \([0,\Tmax)\) and exists
up to \(t_0+\tau_M>\Tmax\).
This contradicts the definition of \(\Tmax\)
and proves the corollary.
\end{proof}

\section*{Acknowledgments}
AI tools were used for textual revision and polishing, and occasionally provided tentative proof directions for reference. All proof strategies and rigorous arguments were independently developed and validated by the authors, who take full responsibility for the manuscript.

\section*{Declaration of competing interest}

The authors declare that they have no known competing financial interests or personal relationships that could have appeared to influence the work reported in this paper.

\section*{Data availability}

Data sharing is not applicable to this article as no data sets were generated or analysed during the current study.

\end{document}